\documentclass[11pt]{amsart}
\usepackage[margin=1in]{geometry}

\usepackage{amssymb}
\usepackage{amsthm}
\usepackage{amsmath}
\usepackage{mathrsfs}
\usepackage{amsbsy}
\usepackage[all]{xy}
\usepackage{bm}
\usepackage{hyperref}
\usepackage{tikz}
\usepackage{array}
\usepackage{float}
\usepackage{enumerate}
\usepackage{xcolor}
\usepackage{hhline}
\allowdisplaybreaks
\usepackage{cite}
\usepackage{tabularx,booktabs}

\newcommand{\tl}{\operatorname{tl}}
\newcommand{\Av}{\operatorname{Av}}
\newcommand{\LRmax}{\operatorname{LRmax}}
\newcommand{\del}{\operatorname{del}}

\newtheorem{theorem}{Theorem}[section]
\newtheorem{proposition}[theorem]{Proposition}

\newtheorem{lemma}[theorem]{Lemma}

\theoremstyle{definition}

\newtheorem{example}[theorem]{Example}
\begin{document}

\title[]{Highly Sorted Permutations and Bell Numbers}
\subjclass[2010]{}

\author[]{Colin Defant}
\address[]{Department of Mathematics, Princeton University, Princeton, NJ 55455}
\email{cdefant@princeton.edu}

\begin{abstract}
Let $s$ denote West's stack-sorting map. For all positive integers $m$ and all integers $n\geq 2m-2$, we give a simple characterization of the set $s^{n-m}(S_n)$; as a consequence, we find that $|s^{n-m}(S_n)|$ is the $m^\text{th}$ Bell number $B_m$. We also prove that the restriction $n\geq 2m-2$ is tight by showing that $|s^{m-3}(S_{2m-3})|=B_m+m-2$ for all $m\geq 3$.   
\end{abstract}

\maketitle

\vspace{-.1cm}
\section{Introduction}

West's stack-sorting map is a function $s$ that sends permutations to permutations; it was defined by West in his dissertation \cite{West} as a deterministic version of a stack-sorting machine introduced by Knuth in \emph{The Art of Computer Programming} \cite{Knuth}. There has been a great deal of interest in $s$ from the point of view of sorting permutations (see \cite{Bona, BonaSurvey, Zeilberger, DefantCounting, DefantTroupes, DefantMonotonicity} and the references therein) since, as is easily verified, $s^{n-1}$ sends every permutation in $S_n$ to the identity permutation $123\cdots n$. The stack-sorting map also has interesting properties that closely link it with other parts of combinatorics such as combinatorial free probability theory (see \cite{DefantCatalan, DefantEngenMiller, DefantTroupes, Hanna} and the references therein).

Many of the classical questions about the stack-sorting map concern the notion of a  \emph{$t$-stack-sortable} permutation, which is a permutation $\pi$ such that $s^t(\pi)$ is increasing. Knuth \cite{Knuth} initiated both the study of stack-sorting and the study of permutation patterns when he showed that a permutation is $1$-stack-sortable if and only if it avoids the pattern $231$. He was also the first to use what is now called the ``kernel method'' (see \cite{Banderier} for details) when he proved that the number of $1$-stack-sortable (i.e., $231$-avoiding) permutations in $S_n$ is the $n^\text{th}$ Catalan number $\frac{1}{n+1}\binom{2n}{n}$. West \cite{West} characterized $2$-stack-sortable permutations and formulated the conjecture, which Zeilberger \cite{Zeilberger} later proved, that the number of $2$-stack-sortable permutations in $S_n$ is $\frac{2}{(n+1)(2n+1)}\binom{3n}{n}$. \'Ulfarsson \cite{Ulfarsson} found a complicated characterization of $3$-stack-sortable permutations in terms of what he called ``decorated patterns,'' but it was not until recently that a polynomial-time algorithm for counting $3$-stack-sortable permutations was found in \cite{DefantCounting}. It is likely that there is no simple formula that enumerates $3$-stack-sortable permutations, and $4$-stack-sortable permutations are probably even more unwieldy. 

In \cite{Bousquet}, Bousquet-M\'elou defined a permutation to be \emph{sorted} if it is in the image of $s$, and she found a recurrence relation that can be used to count sorted permutations. However, the asymptotics of the sequence enumerating sorted permutations is still not well understood; the current author recently proved that the limit $\displaystyle\lim_{n\to\infty}\left(|s(S_n)|/n!\right)^{1/n}$ exists and lies between $0.68631$ and $0.75260$ (the proof of the upper bound makes use of free probability theory and generalized hypergeometric functions). We say a permutation is \emph{$t$-sorted} if it is in the image of $s^t$. The article \cite{DefantDescents} proves that the maximum number of descents that a $t$-sorted permutation of length $n$ can have is $\left\lfloor\frac{n-t}{2}\right\rfloor$ and also characterizes the permutations that achieve this maximum when $n\equiv t\pmod 2$.  

In this paper, we continue the study of $t$-sorted permutations in $S_n$, focusing on their characterization and enumeration when $t$ is close to $n$; in this case, we casually call $t$-sorted permutations of length $n$ \emph{highly sorted}. Our motivation for this line of work comes from the recent article \cite{Asinowski}, which thoroughly explores many aspects of the pop-stack-sorting map $\mathsf{Pop}$ (an interesting variant of West's stack-sorting map). Just as $s^{n-1}(S_n)=\{123\cdots n\}$, it is known (though surprisingly difficult to prove) that $\mathsf{Pop}^{n-1}(S_n)=\{123\cdots n\}$. The paper \cite{Asinowski} gives a very nice characterization of the set $\mathsf{Pop}^{n-2}(S_n)$. For each fixed $m\geq 1$, we will characterize and enumerate the sets $s^{n-m}(S_n)$ for all $n\geq 2m-2$.  

The notion of a $t$-sorted permutation is in some sense dual to that of a $t$-stack-sortable permutation. Hence, another motivation for studying highly sorted permutations comes from results in the literature concerning $t$-stack-sortable permutations of length $n$ for $t$ close to $n$. West \cite{West} showed that a permutation in $S_n$ is $(n-2)$-stack-sortable if and only if it does not end in the suffix $n1$. He also characterized and enumerated the $(n-3)$-stack-sortable permutations in $S_n$. 
Claesson, Dukes, and Steingr\'imsson \cite{Claessonn-4} continued this line of work by characterizing and enumerating $(n-4)$-stack-sortable permutations in $S_n$. In the same article, the authors conjectured that for every fixed $m\geq 1$, there exist positive integers $a_0,\ldots,a_{m-1}$ such that the number of $(n-m)$-stack-sortable permutations in $S_n$ that are not $(n-m-1)$-stack-sortable is $\frac{(m-1)!(n-m-1)!}{(2m-2)!}\sum_{i=0}^{m-1}a_i{n-2m\choose i}$ for all $n\geq 2m$. One can think of Theorem~\ref{Thm1} below as a sort of dual of this conjecture. 

We need just a bit more notation in order to state our main result. Throughout this paper, a \emph{permutation} is an ordering of a finite set of positive integers (for example, we consider $2584$ to be a permutation). We write $S_n$ for the set of permutations of the set $[n]:=\{1,\ldots,n\}$. A \emph{descent} of a permutation $\pi=\pi_1\cdots\pi_n$ is an index $i\in[n-1]$ such that $\pi_i>\pi_{i+1}$. If $i$ is a descent of $\pi$, then we call the entry $\pi_i$ a \emph{descent top} of $\pi$. A \emph{left-to-right maximum} of $\pi$ is an entry $\pi_j$ such that $\pi_j>\pi_\ell$ for all $1\leq \ell<j$; let $\LRmax(\pi)$ denote the set of left-to-right maxima of $\pi$. The \emph{tail length} of a permutation $\pi=\pi_1\cdots\pi_n\in S_n$, denoted $\tl(\pi)$, is the largest integer $\ell\in\{0,\ldots,n\}$ such that $\pi_i=i$ for all $i\in\{n-\ell+1,\ldots,n\}$. For example, $\tl(23145)=2$, $\tl(23154)=0$, and $\tl(12345)=5$. Recall that the $n^\text{th}$ \emph{Bell number} $B_n$ is defined to be the number of set partitions of the set $[n]$. The Bell numbers form the OEIS sequence A000110 \cite{OEIS} and can alternatively be defined via their exponential generating function \[\sum_{n\geq 0}B_n\frac{x^n}{n!}=e^{e^x-1}.\]

\begin{theorem}\label{Thm1}
Let $m$ and $n$ be positive integers such that  $n\geq 2m-2$. A permutation $\pi\in S_n$ is in the image of $s^{n-m}$ if and only if $\tl(\pi)\geq n-m$ and every descent top of $\pi$ is a left-to-right maximum of $\pi$. Consequently, \[|s^{n-m}(S_n)|=B_m.\] 
\end{theorem}

One might ask if the hypothesis $n\geq 2m-2$ in the previous theorem can be replaced by, say, $n\geq 2m-3$. The next theorem shows that it cannot. For $3\leq \ell\leq 2m-3$, we define $\zeta_{\ell,m}$ to be the permutation $\ell21345\cdots(\ell-1)(\ell+1)(\ell+2)\cdots (2m-3)$ in $S_{2m-3}$. This is the permutation obtained by swapping the entries $1$ and $2$ in the identity permutation $123\cdots (2m-3)$ and then moving the entry $\ell$ to the beginning of the permutation. For example, $\zeta_{3,3}=321$, $\zeta_{3,4}=32145$, $\zeta_{4,4}=42135$, and $\zeta_{5,4}=52134$. 

\begin{theorem}\label{Thm2}
Let $m\geq 3$ be an integer. A permutation $\pi\in S_{2m-3}$ is in the image of $s^{m-3}$ if and only if one of the following holds: 
\begin{itemize}
\item $\tl(\pi)\geq m-3$ and every descent top of $\pi$ is a left-to-right maximum of $\pi$;
\item $\pi=\zeta_{\ell,m}$ for some $\ell\in\{3,\ldots,m\}$. 
\end{itemize} 
Consequently, \[|s^{m-3}(S_{2m-3})|=B_m+m-2.\]
\end{theorem}

\section{Preliminaries}
A \emph{permutation} is an ordering of a finite set of positive integers, which we write as a word in one-line notation. Let $S_n$ denote the set of permutations of $[n]$. The \emph{standardization} of a permutation $\pi=\pi_1\cdots\pi_n$ is the permutation obtained by replacing the $i^\text{th}$-smallest entry in $\pi$ with $i$ for all $i$. For example, the standardization of $3869$ is $1324$. We say entries $b,a$ \emph{form an occurrence of the pattern $21$ in $\pi$} if $a<b$ and $b$ appears to the left of $a$ in $\pi$. We say entries $b,c,a$ \emph{form an occurrence of the pattern $231$ in $\pi$} if $a<b<c$ and the entries $b,c,a$ appear in this order in $\pi$.  We will also make use of the barred pattern $32\overline{4}1$. We say $\pi$ \emph{contains} $32\overline{4}1$ if there are indices $i_1<i_2<i_3$ such that $\pi_{i_1}>\pi_{i_2}>\pi_{i_3}$ and such that $\pi_j<\pi_{i_1}$ whenever $i_2<j<i_3$. In this case, we say $\pi_{i_1},\pi_{i_2},\pi_{i_3}$ \emph{form an occurrence of the pattern $32\overline{4}1$ in $\pi$}. If $\pi$ does not contain $32\overline{4}1$, we say it \emph{avoids} $32\overline{4}1$. Let $\Av_n(32\overline{4}1)$ be the set of permutations in $S_n$ that avoid $32\overline{4}1$. 

The importance of the barred pattern $32\overline{4}1$ for our purposes comes from the following result due to Callan. We include its short proof for the sake of completeness. 

\begin{theorem}[\!\!\cite{Callan}]\label{Lem3241}
A permutation $\pi$ avoids $32\overline{4}1$ if and only if every descent top of $\pi$ is a left-to-right maximum of $\pi$. Furthermore, $|\Av_n(32\overline{4}1)|=B_n$.
\end{theorem}

\begin{proof}
Write $\pi=\pi_1\cdots\pi_n$. If $\pi_i$ is a descent top of $\pi$ that is not a left-to-right maximum of $\pi$, then there exists $j<i$ such that $\pi_j>\pi_i$. Then $\pi_j,\pi_i,\pi_{i+1}$ form an occurrence of the pattern $32\overline{4}1$ in $\pi$. 

Conversely, suppose entries $c,b,a$ form an occurrence of $32\overline{4}1$ in $\pi$. Let $b'$ be the rightmost entry that appears between $c$ and $a$ in $\pi$ and is greater than $a$. We cannot have $b'>c$ since $c,b,a$ form an occurrence of $32\overline{4}1$. Therefore, $b'$ is a descent top of $\pi$ that is not a left-to-right maximum of $\pi$. 

Given $\pi\in\Av_n(32\overline{4}1)$, let $j(1)<\cdots<j(\ell)$ be the indices such that $\pi_{j(1)},\ldots,\pi_{j(\ell)}$ are the left-to-right maxima of $\pi$. Let $\beta_r(\pi)=\{\pi_i:j(r)\leq i<j(r+1)\}$. Let $\mathcal B(\pi)=\{\beta_1(\pi),\ldots,\beta_{\ell-1}(\pi)\}$. It is straightforward to check that the map $\mathcal B$ is a bijection from $\Av_n(32\overline{4}1)$ to the set of set partitions of $[n]$. Hence, $|\Av_n(32\overline{4}1)|=B_n$. 
\end{proof}

We now define the stack-sorting map $s$. Assume we are given an input permutation $\pi=\pi_1\cdots\pi_n$. Throughout this procedure, if the next entry in the input permutation is smaller than the entry at the top of the stack or if the stack is empty, the next entry in the input permutation is placed at the top of the stack. Otherwise, the entry at the top of the stack is annexed to the end of the growing output permutation. This procedure stops when the output permutation has length $n$. We then define $s(\pi)$ to be this output permutation. Figure~\ref{Fig1} illustrates this procedure and shows that $s(4162)=1426$.  

\begin{figure}[h]
\begin{center}
\includegraphics[width=1\linewidth]{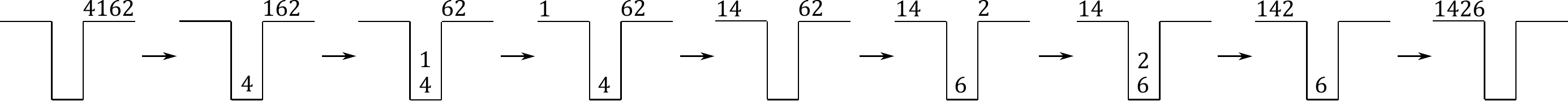}
\caption{The stack-sorting map $s$ sends $4162$ to $1426$.}
\label{Fig1}
\end{center}  
\end{figure}

There is also a simple recursive definition of the map $s$. First, we declare that $s$ sends the empty permutation to itself. Given a nonempty permutation $\pi$, we can write $\pi=LmR$, where $m$ is the largest entry in $\pi$. We then define $s(\pi)=s(L)s(R)m$. For example, 
\[s(5273614)=s(52)\,s(3614)\,7=s(2)\,5\,s(3)\,s(14)\,67=253\,s(1)\,467=2531467.\]

We now collect some basic facts about the stack-sorting map that will be useful in the next section. The following lemma is an immediate consequence of either definition of $s$. 

\begin{lemma}\label{Lem231}
Let $\pi$ be a permutation. Two entries $b,a$ form an occurrence of the pattern $21$ in $s(\pi)$ if and only if there is an entry $c$ such that $b,c,a$ form an occurrence of the pattern $231$ in $\pi$. 
\end{lemma}

Given a nonempty permutation $\pi$, we define $\del_1(\pi)$ to be the permutation obtained by deleting the smallest entry from $\pi$. For example, $\del_1(49628)=4968$. 

\begin{lemma}\label{Lem:del}
If $\pi$ is a nonempty permutation, then $s^t(\del_1(\pi))=\del_1(s^t(\pi))$ for every $t\geq 0$. 
\end{lemma}

\begin{proof}
It suffices to prove the case in which $t=1$; the general case will then follow by induction on $t$. We proceed by induction on the length $n$ of $\pi$, noting first that the proof is trivial if $n=1$. Assume $n\geq 2$. Write $\pi=LmR$, where $m$ is the largest entry in $\pi$. If the smallest entry in $\pi$ is in $L$, then $s(\del_1(\pi))=s(\del_1(L)mR)
=s(\del_1(L))s(R)m=\del_1(s(L))Rm=\del_1(s(L)s(R)m)=\del_1(s(\pi))$, where we have used the induction hypothesis to see that $s(\del_1(L))=\del_1(s(L))$. Similarly, if the smallest entry in $\pi$ is in $R$, then $s(\del_1(\pi))=s(Lm\del_1(R))
=s(L)s(\del_1(R))m=s(L)\del_1(s(R))m=\del_1(s(L)s(R)m)=\del_1(s(\pi))$. 
\end{proof}

\begin{lemma}\label{Lem:Increasing}
Let $\sigma\in S_n$ be a permutation whose last entry is $1$, and let $t\geq 0$. The entries to the right of $1$ in $s^t(\sigma)$ appear in increasing order in $s^t(\sigma)$. 
\end{lemma}

\begin{proof}
The lemma is vacuously true if $t=0$, so we may assume $t\geq 1$ and proceed by induction on $t$. We can write $s^{t-1}(\sigma)=L1R$, where $R$ is an increasing permutation. Suppose by way of contradiction that there are entries $a,b$ appearing to the right of $1$ in $s^t(\sigma)$ such that $b$ appears to the left of $a$ and $a<b$. Then $b,a$ form an occurrence of the pattern $21$ in $s^t(\sigma)$, so it follows from Lemma~\ref{Lem231} that there is an entry $c$ such that $b,c,a$ form an occurrence of $231$ in the permutation $s^{t-1}(\sigma)=L1R$. Notice that $c$ must be in $L$ because $R$ is an increasing permutation. However, this means that $b,c,1$ form an occurrence of the pattern $231$ in $s^{t-1}(\sigma)$, so $b,1$ form an occurrence of the pattern $21$ in $s^t(\sigma)$. This contradicts the fact that $b$ appears to the right of $1$ in $s^t(\sigma)$.   
\end{proof}

The next lemma follows immediately from a simple inductive argument and the definition of $s$; it is the reason why $s^{n-1}(S_n)=\{123\cdots n\}$.  

\begin{lemma}\label{Lem:tl}
If $\sigma\in S_n$ and $t\geq 0$, then $\tl(s^t(\sigma))\geq t$.
\end{lemma}

Theorems~\ref{Thm1} and \ref{Thm2}, which we prove in the next section, determine the sizes of $s^{n-m}(S_n)$ when $n\geq 2m-3$. We end this section with a simple proposition that gives some information about the sizes of $s^{n-m}(S_n)$ for all $n\geq m$ when $m$ is fixed. 

\begin{proposition}\label{Prop2}
Fix a positive integer $m$. The sequence $(|s^{n-m}(S_n)|)_{n\geq m}$ is nonincreasing. 
\end{proposition} 

\begin{proof}
Suppose $n\geq m+1$ and $\pi\in s^{n-m}(S_n)$. Let $\sigma$ be such that $s^{n-m}(\sigma)=\pi$. We can write $\pi=\pi^*n$ and $s(\sigma)=\tau n$ for some $\pi^*,\tau\in S_{n-1}$. We have \[\pi^*n=\pi=s^{n-m}(\sigma)=s^{n-m-1}(\tau n)=s^{n-m-1}(\tau)n,\] so $\pi^*=s^{n-m-1}(\tau)\in s^{(n-1)-m}(S_{n-1})$. This shows that the map $\pi\mapsto \pi^*$ is an injection from $s^{n-m}(S_n)$ to $s^{(n-1)-m}(S_{n-1})$.  
\end{proof}

Theorem~\ref{Thm2} and Proposition~\ref{Prop2} tell us that $|s^{n-m}(S_n)|\geq B_m+m-2$ whenever $m\leq n\leq 2m-3$. 

\section{Proofs of the Main Theorems}

We now establish some results that will lead up to the proofs of Theorems~\ref{Thm1} and \ref{Thm2}. Recall that $\LRmax(\pi)$ denotes the set of left-to-right maxima of a permutation $\pi$. The reader may find it helpful to refer to Example~\ref{Exam1} for an illustration of the next lemma's proof. 

\begin{lemma}\label{Lem1}
If $\pi$ is a permutation that avoids $32\overline{4}1$ and ends in its largest entry, then there exists $\sigma\in s^{-1}(\pi)$ such that $\sigma$ avoids $32\overline{4}1$ and $\LRmax(\sigma)=\LRmax(\pi)$. 
\end{lemma}

\begin{proof}
Let $n$ be the length of $\pi$, and write $\pi=\pi_1\cdots\pi_n$. The lemma is trivial if $n=1$, so we may assume $n\geq 2$ and proceed by induction on $n$. The lemma is true for $\pi$ if and only if it is true for the standardization of $\pi$, so we may assume without loss of generality that $\pi\in S_n$. Let $r$ be the index such that $\pi_r=1$. Let $\pi'=\del_1(\pi)$. Since $\pi'$ is a permutation of length $n-1$ that avoids $32\overline{4}1$ and ends in its largest entry, it follows by induction that there exists $\sigma'\in s^{-1}(\pi')$ such that $\sigma'$ avoids $32\overline{4}1$ and $\LRmax(\sigma')=\LRmax(\pi')$. Suppose for the moment that $r=1$, and let $\sigma=1\sigma'$. It follows immediately from the definition of $s$ that $s(\sigma)=1s(\sigma')=1\pi'=\pi$. Furthermore, $\sigma$ avoids $32\overline{4}1$ because $\sigma'$ does. Finally, \[\LRmax(\sigma)=\{1\}\cup\LRmax(\sigma')=\{1\}\cup\LRmax(\pi')=\LRmax(\pi).\] 

Now assume $r\geq 2$, and let $a=\pi_{r-1}$. Since $\pi$ avoids $32\overline{4}1$ and $a$ is a descent top of $\pi$, we know by Theorem~\ref{Lem3241} that $a\in\LRmax(\pi)=\LRmax(\pi')=\LRmax(\sigma')$. We assumed that $\pi$ ends in its largest entry, so $a$ is not the largest entry of $\pi$. Therefore, $a$ is not the largest entry of $\sigma'$. Because $a\in\LRmax(\sigma')$, there must be an entry to the right of $a$ in $\sigma'$ that is larger than $a$; among all such entries, let $b$ be the one that is farthest to the left in $\sigma'$. Let $\sigma$ be the permutation obtained by inserting the entry $1$ immediately after $b$ in $\sigma'$. Note that $\LRmax(\sigma)=\LRmax(\sigma')=\LRmax(\pi)$. In particular, $a\in\LRmax(\sigma)$. Since $b$ is the leftmost entry in $\sigma$ that is larger than $a$ and to the right of $a$, we must have $b\in\LRmax(\sigma)$. Because $\sigma'$ avoids $32\overline{4}1$, we know by Theorem~\ref{Lem3241} that every descent top of $\sigma'$ is in $\LRmax(\sigma')$. Every descent top of $\sigma$, except possibly $b$, is a descent top of $\sigma'$. It follows that every descent top of $\sigma$ is in $\LRmax(\sigma)$, so $\sigma$ avoids $32\overline{4}1$. 

We are left to prove that $s(\sigma)=\pi$. Imagine applying the stack-sorting procedure to $\sigma$. Because $a$ is a left-to-right maximum of $\sigma$, it will never sit on top of any entries in the stack. By the choice of $b$, there must be a point in time during the procedure when $b$ is next in line to enter the stack and $a$ is the only entry in the stack. In the next steps in the procedure, $a$ is popped out, $b$ is pushed in, $1$ is pushed in, and then $1$ is popped out. It follows that $1$ appears immediately to the right of $a$ in $s(\sigma)$. Recall that $1$ also appears immediately to the right of $a$ in $\pi$. We know that $\del_1(\pi)=\pi'=s(\sigma')=s(\del_1(\sigma))=\del_1(s(\sigma))$ by Lemma~\ref{Lem:del}. Therefore, $s(\sigma)=\pi$.
\end{proof}

\begin{example}\label{Exam1}
Let us give a concrete example of the proof of Lemma~\ref{Lem1}. Suppose $\pi=527148369$. We have $r=4$ and $\pi'=52748369$. We can take $\sigma'=57284936$ since this permutation avoids $32\overline 41$, has the same left-to-right maxima as $\pi$, and satisfies $s(\sigma')=\pi'$. Because $r\geq 2$, we put $a=\pi_{r-1}=7$. The entries that are larger than $a$ and to the right of $a$ in $\sigma'$ are $8$ and $9$; we choose $b$ to be the one that is farthest to the left in $\sigma'$, which is $8$. Then $\sigma=572814936$. Observe that $\sigma$ does indeed avoid $32\overline 41$ and satisfy $s(\sigma)=\pi$. \hfill$\lozenge$
\end{example}

\begin{proposition}\label{Prop1}
If $\pi\in \Av_n(32\overline{4}1)$ has tail length $\ell$, then $\pi\in s^\ell(\Av_n(32\overline{4}1))$. 
\end{proposition}

\begin{proof}
The statement is trivial if $\ell=0$, and it is immediate from Lemma~\ref{Lem1} if $\ell=1$. Let us now assume $\ell\geq 2$ and proceed by induction on $\ell$. Let $\pi^*$ be the permutation in $\Av_{n-1}(32\overline{4}1)$ obtained by removing the entry $n$ from $\pi$. Since $\pi^*$ has tail length $\ell-1$, it follows by induction that there exists $\tau\in\Av_{n-1}(32\overline{4}1)$ such that $s^{\ell-1}(\tau)=\pi^*$. Now consider the permutation $\tau n$, which avoids $32\overline{4}1$ and ends in its largest entry. By Lemma~\ref{Lem1}, there exists $\sigma\in\Av_n(32\overline{4}1)$ such that $s(\sigma)=\tau n$. We have $s^\ell(\sigma)=s^{\ell-1}(\tau n)=s^{\ell-1}(\tau)n=\pi^* n=\pi$, as desired. 
\end{proof}

\begin{lemma}\label{Lem2}
Suppose $\pi\in S_n$ contains an occurrence of the pattern $32\overline{4}1$ that involves the entry $1$, and suppose $\sigma\in s^{-1}(\pi)$. Then $\sigma$ also contains an occurrence of the pattern $32\overline{4}1$ that involves the entry $1$. Furthermore, there exist distinct entries $c,d\in\{3,\ldots,n\}$ that appear to the left of $1$ in $\sigma$ and appear to the right of $1$ in $\pi$.   
\end{lemma}

\begin{proof}
Write $\pi=\pi_1\cdots\pi_n$. Let $r$ be the index such that $\pi_r=1$. Let $a=\pi_{r-1}$. Since  $\pi$ contains an occurrence of the pattern $32\overline{4}1$ that involves the entry $1$, there must be an entry $b>a$ that appears to the left of $a$ in $\pi$. Since $b,a$ form an occurrence of the pattern $21$ in $\pi$, it follows from Lemma~\ref{Lem231} that there is an entry $c>b$ such that $b,c,a$ form an occurrence of the pattern $231$ in $\sigma$; among all such entries $c$, choose the one that appears farthest to the right in $\sigma$. This choice of $c$ implies that there are no entries between $c$ and $a$ in $\sigma$ that are greater than $c$. It follows that $c$ appears to the right of $a$ in $s(\sigma)=\pi$. Because $\pi_{r-1}=a$ and $\pi_r=1$, the entry $c$ must appear to the right of $1$ in $\pi$. By Lemma~\ref{Lem231}, there cannot be an entry that is greater than $c$ and lies between $c$ and $1$ in $\sigma$. Hence, $c,a,1$ form an occurrence of the pattern $32\overline{4}1$ in $\sigma$.   

To prove the second statement of the lemma, let $d$ be the rightmost entry in $\sigma$ that is greater than $a$ and lies between $a$ and $1$ in $\sigma$; note that such an entry necessarily exists because $a,1$ form an occurrence of the pattern $21$ in $\pi$. Since $c$ and $d$ are both greater than $a$, which is itself greater than $1$, we must have $c,d\in\{3,\ldots,n\}$. Furthermore, $c$ and $d$ appear to the left of $1$ in $\sigma$. We saw above that $c$ lies to the right of $1$ in $\pi$. Our choice of $d$ implies that there are no entries greater than $d$ lying between $d$ and $1$ in $\sigma$, so $d$ also lies to the right of $1$ in $\pi$. 
\end{proof}

We can now complete the proofs of Theorems~\ref{Thm1} and \ref{Thm2}. 

\begin{proof}[Proof of Theorem~\ref{Thm1}] 
First, suppose that $\tl(\pi)\geq n-m$ and that every descent top of $\pi$ is a left-to-right maximum of $\pi$. Theorem~\ref{Lem3241} tells us that $\pi\in \Av_n(32\overline{4}1)$, so it follows from Proposition~\ref{Prop1} that $\pi\in s^{\tl(\pi)}(\Av_n(32\overline{4}1))\subseteq s^{\tl(\pi)}(S_n)\subseteq s^{n-m}(S_n)$. 

We now prove that if $n\geq 2m-2$ and $\pi\in s^{n-m}(S_n)$, then $\tl(\pi)\geq n-m$ and $\pi\in\Av_n(32\overline{4}1)$ (again, the latter condition is equivalent to the condition that every descent top of $\pi$ is a left-to-right maximum of $\pi$). When $m=1$, the claim is easy since $s^{n-1}(S_n)=\{123\cdots n\}$. We now assume $m\geq 2$ and proceed by induction on $m$. Suppose $\pi=s^{n-m}(\sigma)$ for some $\sigma\in S_n$. We already know from Lemma~\ref{Lem:tl} that $\tl(\pi)\geq n-m$, so we need only show that $\pi$ avoids $32\overline{4}1$. Let $\widehat\pi$ and $\widehat\sigma$ be the standardizations of $\del_1(\pi)$ and $\del_1(\sigma)$, respectively. By Lemma~\ref{Lem:del}, we have \[\del_1(\pi)=\del_1(s^{n-m}(\sigma))=s^{n-m}(\del_1(\sigma)),\] so $\widehat\pi=s^{n-m}(\widehat\sigma)$. This shows that $\widehat\pi\in s^{(n-1)-(m-1)}(S_{n-1})$, so it follows by induction that $\widehat\pi$ avoids $32\overline{4}1$. Consequently, $\del_1(\pi)$ avoids $32\overline{4}1$. To complete the proof, it suffices to show that $\pi$ does not contain an occurrence of the pattern $32\overline{4}1$ that involves the entry $1$. Assume by way of contradiction that 
$\pi$ does contain an occurrence of the pattern $32\overline{4}1$ that involves the entry $1$. For $0\leq i\leq n-m$, let $\sigma^{(i)}=s^{n-m-i}(\sigma)$. Thus, $\sigma^{(0)}=\pi$ and $\sigma^{(n-m)}=\sigma$. Since $\sigma^{(i)}\in s^{-1}(\sigma^{(i-1)})$, we can use Lemma~\ref{Lem2} and induction on $i$ to see that each of the permutations $\sigma^{(i)}$ contains an occurrence of the pattern $32\overline{4}1$ that involves the entry $1$. Let $r_i$ be the index such that the $r_i^\text{th}$ entry in $\sigma^{(i)}$ is $1$. Note that every entry appearing to the right of $1$ in $\sigma^{(i+1)}$ also appears to the right of $1$ in $\sigma^{(i)}$. Furthermore, Lemma~\ref{Lem2} tells us that there are at least $2$ entries that lie to the left of $1$ in $\sigma^{(i+1)}$ and lie to the right of $1$ in $\sigma^{(i)}$. It follows that $r_i\leq r_{i+1}-2$ for all $0\leq i\leq n-m-1$. Consequently, $r_0\leq r_{n-m}-2(n-m)\leq n-2(n-m)=2m-n\leq 2m-(2m-2)=2$. This says that $1$ is either the first or second entry of $\pi$, which contradicts our assumption that $\pi$ contains an occurrence of the pattern $32\overline{4}1$ involving $1$. 

We have shown that a permutation $\pi\in S_n$ is in the image of $s^{n-m}$ if and only if it avoids $32\overline{4}1$ and has tail length at least $n-m$. It follows that $|s^{n-m}(S_n)|=|\Av_m(32\overline{4}1)|$, and we know by Theorem~\ref{Lem3241} that $|\Av_m(32\overline{4}1)|$ is the Bell number $B_m$. 
\end{proof}

\begin{proof}[Proof of Theorem~\ref{Thm2}]
Let $\pi\in S_{2m-3}$ for some $m\geq 3$. If $\tl(\pi)\geq m-3$ and every descent top of $\pi$ is a left-to-right maximum of $\pi$, then $\pi\in \Av_{2m-3}(32\overline{4}1)$, so it follows from Proposition~\ref{Prop1} that $\pi\in s^{\tl(\pi)}(\Av_{2m-3}(32\overline{4}1))\subseteq s^{\tl(\pi)}(S_{2m-3})\subseteq s^{m-3}(S_{2m-3})$. Now suppose $\pi=\zeta_{\ell,m}$ for some $\ell\in\{3,\ldots,m\}$. Let \[\xi_{\ell,m}=\ell(m+1)(m+2)\cdots (2m-3)23\cdots(\ell-1)(\ell+1)\cdots m1.\] In other words, $\xi_{\ell,m}$ is obtained from the permutation $(m+1)(m+2)\cdots (2m-3)123\cdots m$ by moving the entry $\ell$ to the beginning and moving the entry $1$ to the end. It is straightforward to check that $s^{m-3}(\xi_{\ell,m})=\zeta_{\ell,m}$, so $\pi=\zeta_{\ell,m}$ is in the image of $s^{m-3}$. 

To prove the converse, suppose $\pi=s^{m-3}(\sigma)$ for some $\sigma\in S_{2m-3}$. We know by Lemma~\ref{Lem:tl} that $\tl(\pi)\geq m-3$. Now suppose that not every descent top of $\pi$ is in $\LRmax(\pi)$; we will prove that $\pi=\zeta_{\ell,m}$ for some $\ell\in\{3,\ldots,m\}$. Let $\widehat\pi$ and $\widehat\sigma$ be the standardizations of $\del_1(\pi)$ and $\del_1(\sigma)$, respectively. By Lemma~\ref{Lem:del}, we have \[\del_1(\pi)=\del_1(s^{m-3}(\sigma))=s^{m-3}(\del_1(\sigma)),\] so $\widehat\pi=s^{m-3}(\widehat\sigma)$. This shows that $\widehat\pi\in s^{(2m-4)-(m-1)}(S_{2m-4})$, so Theorem~\ref{Thm1} guarantees that $\widehat\pi$ avoids $32\overline{4}1$. Consequently, $\del_1(\pi)$ avoids $32\overline{4}1$. Since not every descent top of $\pi$ is in $\LRmax(\pi)$, it must be the case that
$\pi$ contains an occurrence of the pattern $32\overline{4}1$ that involves the entry $1$. 

For $0\leq i\leq m-3$, let $\sigma^{(i)}=s^{m-3-i}(\sigma)$. Thus, $\sigma^{(0)}=\pi$ and $\sigma^{(m-3)}=\sigma$. Since $\sigma^{(i)}\in s^{-1}(\sigma^{(i-1)})$, we can use Lemma~\ref{Lem2} and induction on $i$ to see that each of the permutations $\sigma^{(i)}$ contains an occurrence of the pattern $32\overline{4}1$ that involves the entry $1$. Let $R_i$ be the set of entries that appear to the right of $1$ in $\sigma^{(i)}$. Note that $R_0\supseteq R_1\supseteq\cdots\supseteq R_{m-3}$. Furthermore, Lemma~\ref{Lem2} tells us that there are at least $2$ elements of $\{3,\ldots,2m-3\}$ in $R_i\setminus R_{i+1}$ for each $i\in\{0,\ldots,m-4\}$. This proves that there are at least $2m-6$ elements of $\{3,\ldots,2m-3\}$ in $R_0\setminus R_{m-3}$. Since $\pi$ has length $2m-3$ and contains an occurrence of $32\overline{4}1$ that involves the entry $1$, we must have $\pi_1>\pi_2>\pi_3=1$. Furthermore, there must be exactly $2m-6$ entries to the right of $1$ in $\pi$, so \[2m-6\leq |(R_0\setminus R_{m-3})\cap\{3,\ldots,2m-3\}|\leq|R_0\setminus R_{m-3}|\leq|R_0|=2m-6.\] These inequalities must, of course, be equalities, so $R_{m-3}=\emptyset$ and $R_0\subseteq \{3,\ldots,m-3\}$. The fact that $R_{m-3}$ is empty tells us that $\sigma$ ends in the entry $1$. Since $\pi=s^{m-3}(\sigma)$, it follows from Lemma~\ref{Lem:Increasing} that all of the entries to the right of $1$ in $\pi$ appear in increasing order. The fact that $R_0\subseteq\{3,\ldots,m-3\}$ tells us that $2$ appears to the left of $1$ in $\pi$. Putting this all together, we find that $\pi=\zeta_{\ell,m}$ for some $\ell\in\{3,\ldots,2m-3\}$. We have seen that the tail length of $\pi$ is at least $m-3$, so we must actually have $\ell\in\{3,\ldots,m\}$, as desired. 
\end{proof}

\section{Open Problems}

We saw in Proposition~\ref{Prop2} that for each fixed $m\geq 1$, the sequence $(|s^{n-m}(S_n)|)_{n\geq m}$ is nonincreasing. It is clear that the first term in this sequence is $|s^0(S_m)|=|S_m|=m!$. The second term in the sequence is $|s(S_{m+1})|$, the number of sorted permutations in $S_{m+1}$; as mentioned in the introduction, the asymptotics of $|s(S_{m+1})|$ for large $m$ are not precisely known. In this article, we showed that $|s^{n-m}(S_n)|=B_m$ for all $n\geq 2m-2$ and that $|s^{n-m}(S_n)|=B_m+m-2$ when $n=2m-3$. What can be said about $|s^{m-4}(S_{2m-4})|$? What can be said asymptotically about the approximate sizes of the numbers $|s^{n-m}(S_n)|$ for $m\leq n\leq 2m-2$ when $m$ is large? It would be interesting to have a more precise understanding of how the sequence $(|s^{n-m}(S_n)|)_{n=m}^{2m-2}$ decreases from $m!$ to $B_m$.   

\section{Acknowledgments}
The author was supported by a Fannie and John Hertz Foundation Fellowship and an NSF Graduate Research Fellowship (grant no. DGE-1656466).

\end{document}